\documentclass[12pt,reqno]{amsart}
\usepackage[margin = 1in]{geometry}
\usepackage{float}
\usepackage{graphicx} 
\usepackage{amsmath}
\usepackage{amssymb}
\usepackage[]{amsthm}
\usepackage{enumitem}
\usepackage{microtype}
\usepackage{url}
\usepackage{tikz,tikz-3dplot} 
\usepackage{verbatim}
\usepackage[
  colorlinks=true,
  linkcolor=black,
  citecolor=black,
  urlcolor=black
]{hyperref}

\newtheoremstyle{case}{}{}{}{}{}{:}{}{}

\theoremstyle{case}

\theoremstyle{definition}
\newtheorem{definition}{Definition}[section]
\newtheorem{theorem}{Theorem}[section]
\newtheorem{lemma}[theorem]{Lemma}
\newtheorem{prop}{Proposition}[section]
\newtheorem{cor}[theorem]{Corollary}

\theoremstyle{definition}
\newtheorem{example}{Example}[section]
\newtheorem{remark}{Remark}[section]
\newtheorem{question}{Question}[section]

\DeclareMathOperator{\conv}{conv}
\DeclareMathOperator{\vol}{vol}
\DeclareMathOperator{\Vol}{Vol}

\newcommand{\assocpoly}[1]{f_{#1}}

\title[Triangulation, Volume, and Symmetries of the Dihedral Permutohedron]{Triangulation, Volume, and Symmetries of a Dihedral Group Analogue of the Permutohedron}
\author{Colin Olairez}
\date{August 2026}

\begin{document}

\begin{abstract}

We study an analogue of the permutohedron for the dihedral group, which we call the dihedral permutohedron. We determine the facet structure of the dihedral permutohedron, use the facet structure to construct a triangulation, and determine a nice combinatorial formula for the volume of the dihedral permutohedron. Lastly, we find a family of interesting unimodular affine symmetries of the odd dihedral permutohedron. 

\end{abstract}

\maketitle
    
\section{Introduction}

Two polytopes associated to the symmetric group have long been of mathematical interest. These two polytopes are the Birkhoff polytope and its projection, the permutohedron. 
The Birkhoff polytope $B_n$ is the convex hull of the $n \times n$ permutation matrices, where an $n \times n$ permutation matrix is embedded as a vector in $\mathbb{R}^{n^2}$. One can obtain the permutohedron from the Birkhoff polytope by setting the basepoint $\rho = (1,2, \cdots, n-1, n)$ and considering 
\[\conv (\{P(\rho) \mid P \ \text{is a permutation matrix} \}),\]
a projection of the Birkhoff polytope to $\mathbb{R}^n$ \cite{schoute,MR45168}. 
The combinatorial structure and geometry of these polytopes are fascinating and there are many papers that discuss this. One can form analogous polytopes for other finite groups. Shmuel Onn in \cite{MR1239510} defines how given any group $G$ of permutation matrices, one can form the \emph{permutation polytope} $\conv(G) \subset \mathbb{R}^{n^2}$ and for any vector $\rho \in \mathbb{R}^n$, the \emph{orbit polytope} $\conv(G\rho) \subset \mathbb{R}^n$. These polytopes are of great mathematical interest in the combinatorial literature.  Onn's paper writes extensively about a subclass of permutation polytopes called the Young permutation polytopes, arising from the Young representations of the symmetric group. In \cite{MR3436389}, affine symmetries of general orbit polytopes are studied. This paper specifically discovers a family of unimodular affine symmetries of the dihedral permutohedron, an analogue of the permutohedron for the dihedral group, for all odd $n \geq 3$. One can see \cite{MR2538016} for more work on permutation polytopes. This paper focuses on the orbit polytope corresponding to the dihedral group \begin{equation}\label{eq:dihedral}D_{2n} := \langle s_1,s_2 \mid s_1^2 = 1, s_2^2 = 1, (s_1s_2)^n = 1 \rangle.\end{equation} 

Since $D_{2n}$ can be thought of as a subgroup of $S_n$ by considering the action of $D_{2n}$ on the $n$ vertices of a regular $n$-gon, one can represent $D_{2n}$ as a subgroup of the $n \times n$ permutation matrices. Explicitly, we derive the matrix representations of $s_1, s_2$ below and construct $D_{2n}$ as a group of permutation matrices with the same generators and relations as before.

\begin{equation}\label{eq:dihedral_matrices}
s_1 \;=\;
\begin{pmatrix}
0 & 1 &        &         &        &        &        \\
1 & 0 &        &         &        &        &        \\
  &   & 0      & 0       & \cdots & 0      & 1      \\
  &   & 0      & 0       & \cdots & 1      & 0      \\
  &   & \vdots &         & \reflectbox{$\ddots$} &        & \vdots \\
  &   & 0      & 1       & \cdots & 0      & 0      \\
  &   & 1      & 0       & \cdots & 0      & 0
\end{pmatrix},
\qquad
s_2 \;=\;
\begin{pmatrix}
1 &        &        &        &        &        \\
  & 0      & 0      & \cdots & 0      & 0 &  1 \\
  & 0      & 0      & \cdots & 0      & 1 & 0  \\
  & 0      & 0      & 0 &  1     & 0 & 0       \\
  & \vdots & \vdots       & \reflectbox{$\ddots$} &  \vdots      & \vdots & \vdots \\
  & 0      & 1      & \cdots & 0      & 0  & 0    \\
  & 1      & 0      & \cdots & 0      & 0 & 0
\end{pmatrix}.
\end{equation}
In these matrix representations, any blank entries are $0$ and the matrices in the bottom right have $1$'s only along the antidiagonal. These generators will also be presented as permutations in cycle notation in the next section. With this matrix perspective we can define the \emph{dihedral polytope}, $\conv(D_{2n}) \subset \mathbb{R}^{n^2}$, which has been studied in previous work, leading to a full understanding of its Ehrhart polynomial and combinatorial structure \cite{MR3029450}. However, the orbit polytope analogue of the permutohedron for the dihedral group, $Q_n := \conv(D_{2n} \rho)$, has hardly any literature on it. In this paper, the main result is a formula for the volume of this dihedral group analogue $Q_n$, which we call the \emph{dihedral permutohedron}. Specifically, this paper proves that the normalized volume of the dihedral permutohedron for $n \geq 3$ is \[
\vol(Q_n) = n^{n-2}\binom{n-1}{\lfloor \frac{n-1}{2} \rfloor}.\] On the way to proving this formula, the facet structure of the dihedral permutohedron is derived as well as a nice way to triangulate the dihedral permutohedron for all $n \geq 3$.

\section{The Dihedral Permutohedron for \texorpdfstring{$n = 3,4$}{n = 3,4}}

For $n = 3$, $D_{2n} = D_6 \cong S_3$, so the dihedral permutohedron for $n=3$ is also the permutohedron for $n = 3$. See the figure of the dihedral permutohedron and a triangulation below. Note that every triangle in the triangulation contains $r_1$, which is a pattern that will continue for all odd $n \geq 3$. For $n > 3$, $D_{2n}$ is a proper subgroup of $S_n$, so the dihedral permutohedron gets more interesting.

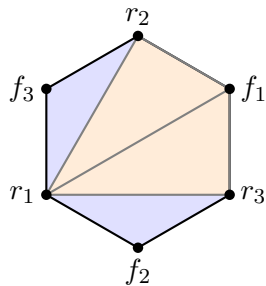
\begin{figure}[h]
\centering
\begin{tikzpicture}[scale=1.4, every node/.style={font=\small}]
  \coordinate (r1) at (210:1);
  \coordinate (f2) at (270:1);
  \coordinate (r3) at (330:1);
  \coordinate (f1) at (30:1);
  \coordinate (r2) at (90:1);
  \coordinate (f3) at (150:1);

  \fill[blue!12] (r1) -- (f2) -- (r3) -- cycle;
  \fill[blue!12] (r2) -- (f3) -- (r1) -- cycle;
  \fill[orange!15] (r1) -- (r3) -- (f1) -- cycle;
  \fill[orange!15] (r1) -- (f1) -- (r2) -- cycle;

  \draw[thick] (r1) -- (f2) -- (r3) -- (f1) -- (r2) -- (f3) -- cycle;
  \draw[thick, gray] (r1) -- (r3);
  \draw[thick, gray] (r3) -- (f1);
  \draw[thick, gray] (f1) -- (r2);
  \draw[thick, gray] (r2) -- (r1);
  \draw[thick, gray] (r1) -- (f1);

  \foreach \v/\lab/\pos in {r1/$r_1$/left, r2/$r_2$/above, r3/$r_3$/right, f1/$f_1$/right, f2/$f_2$/below, f3/$f_3$/left}{
    \fill (\v) circle (1.4pt);
    \node[\pos] at (\v) {\lab};
  }
\end{tikzpicture}
\caption{The dihedral permutohedron for $n=3$: a regular hexagon with vertices alternating between rotations $r_1,r_2,r_3$ and reflections $f_1,f_2,f_3$. The shading shows a triangulation in which the two lightly shaded triangles have normalized volume $1$, and the two darker triangles have normalized volume $2$, yielding an overall normalized volume of $6$, which agrees with the formula of this paper.}
\end{figure}

In the dihedral permutohedron for $n = 4$, every facet has an equal number of rotations and reflections, which is a pattern that we will prove holds for all $n \geq 3$. See the figure below. One tetrahedron of a triangulation of the dihedral permutohedron for $n = 4$ is visible and contains the two vertices $r_1,f_1$. We will see later that one can construct a triangulation for all even $n > 3$ where every simplex contains two vertices $r_1,f_1$. The pattern in the facets will be useful for proving the existence of this triangulation for the even dihedral permutohedron in Section 4. These examples and others were key in making the conjectures that we prove in Section 4.

\begin{figure}[H]
\centering
\begin{tikzpicture}[
  x={(1.65cm,0.28cm)},
  y={(0.42cm,1.06cm)},
  z={(-0.95cm,0.95cm)},
  scale=1, every node/.style={font=\normalsize}
]
  \coordinate (r1) at (0,0,0);
  \coordinate (f2) at (1.4142,0,0);
  \coordinate (f3) at (0,1,0);
  \coordinate (f4) at (0,0,1.4142);
  \coordinate (r4) at (1.4142,1,0);
  \coordinate (r3) at (1.4142,0,1.4142);
  \coordinate (r2) at (0,1,1.4142);
  \coordinate (f1) at (1.4142,1,1.4142);
  
  \draw[dashed] (r1) -- (f3);
  \draw[thick] (r1) -- (f4);
  \draw[orange, thick] (f2) -- (r1);

  \draw[dashed] (f1) -- (r3);
  \draw[dashed] (f2) -- (r3);
  \draw[orange, thick] (f2) -- (r4);
  \draw[thick] (f3) -- (r2);
  \draw[dashed] (f3) -- (r4);
  \draw[dashed] (f4) -- (r3);
  \draw[thick] (f4) -- (r2);
  \draw[thick] (r2) -- (f1);

  \draw[orange, thick] (r1) -- (f1);
  \draw[orange, dashed, thick,] (r1) -- (r4);
  \draw[orange, thick] (f1) -- (f2);
  \draw[orange, thick] (f1) -- (r4);
  
  \foreach \v/\lab/\pos in {%
    r1/$r_1$/left, r2/$r_2$/above, r3/$r_3$/right, r4/$r_4$/below right,%
    f1/$f_1$/above right, f2/$f_2$/below right, f4/$f_4$/above left%
  }{
    \fill (\v) circle (1pt);
    \node[\pos] at (\v) {\lab};
  }
  \fill (f3) circle (1pt);
  \node[left=6pt] at (f3) {$f_3$};
\end{tikzpicture}
\caption{The dihedral permutohedron for $n=4$: a rectangular box with $8$ vertices, alternating between rotations $r_1,\ldots,r_4$ and reflections $f_1,\ldots,f_4$. One of the tetrahedra in a triangulation of the dihedral permutohedron is shown in orange. The dihedral permutohedron for $n=4$ has normalized volume $48$, which is consistent with the formula of this paper.}
\end{figure}
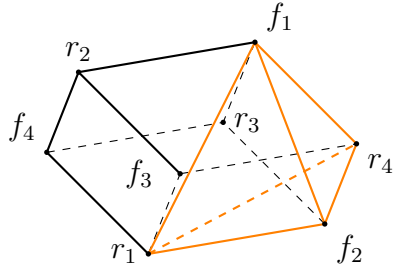
    
\section{Notation and Preliminaries}

We first review existing theory on the dihedral polytope and the dihedral permutohedron. As in the introduction, we can think about $D_{2n}$ as a standalone group or think about $D_{2n}$ as a subgroup of $S_n$, acting on the $n$ vertices of a regular $n$-gon. Just like the permutohedron, we will work with the basepoint $\rho := (1, 2 , \cdots, n)$. We will first define the two permutations that represent the simple reflections that generate the dihedral group. Following the group presentation for $D_{2n}$ in  \eqref{eq:dihedral}, $n = 2m + 1$, for odd $n \geq 3$, we can look at how reflections of the dihedral group act on the $n$ labeled vertices of a regular $n$-gon and derive the following equations for the generating simple reflections in terms of permutations:

\begin{equation}
\label{eq:odd_perm}
\begin{aligned}
s_1 &= (1,2)(3,n)(4,n-1)\cdots(m+1,m+3), \\
s_2 &= (2,n)(3,n-1)(4,n-2)\cdots(m+1,m+2).
\end{aligned}
\end{equation} For even $n > 3$, we can define the simple reflections to be

\begin{equation}
\label{eq:even_perm}
\begin{aligned}
s_1 &= (1,2)(3,n)(4,n-1)\cdots(m+1,m+2), \\
s_2 &= (2,n)(3,n-1)(4,n-2)\cdots(m,m+2).
\end{aligned}
\end{equation} These permutations agree with the matrix presentations of $s_1$ and $s_2$ in \eqref{eq:dihedral_matrices}. Lastly, we have that $\mathbf{1} = (1, 1, \cdots, 1)$, the all $1$'s vector and that given a permutation $\sigma$, $\sigma(1, \cdots, n) = (\sigma(1), \cdots, \sigma(n))$.

\begin{definition}
    We define the \emph{dihedral permutohedron} $Q_n := \conv(D_{2n}\rho)$.
\end{definition}

\begin{lemma}
    Let $Q_n = \conv(D_{2n} \rho)$ be the dihedral permutohedron. The vertices of $Q_n$ are precisely the vectors $w \rho$ with $w \in D_{2n}$, where we use the matrix representation of $w$ from the generators in  \eqref{eq:dihedral_matrices}.
\end{lemma}
\begin{proof}
    At least one of the vectors $w_0\rho$ in the convex hull must be a vertex. Consider any other vector $w_1\rho$. We have that $w_1(w_0^{-1}(w_0\rho)) = w_1\rho$. This is an affine symmetry of the dihedral permutohedron that sends vertices to vertices, so $w_1\rho$ is also a vertex. Indeed, the vertices are exactly $w\rho$ for all $w \in D_{2n}$.
\end{proof}

Fix the notation that 
\begin{equation}\label{eq:rotations}
    r_i := ((1 + (i-1), \cdots, n+(i-1)) \bmod n) \in \mathbb{R}^n.
\end{equation} Here, we choose $n$ as the representative of the identity element in $\mathbb{Z}/n$. These $r_i$ represent the $n$ vertices of the dihedral permutohedron that correspond to the $n$ rotations of the dihedral group. Consider the $n$ vertices of the dihedral permutohedron corresponding to the $n$ rotations. If we choose the $n$ rotation vertices to be the columns of an $n \times n$ matrix, with a specific ordering of the column vectors, they form a \emph{circulant} matrix, a square matrix where each column can be obtained from the next via a cyclic shift by one element. This is the matrix \begin{equation}\label{eq:circ}
  R \;=\;
  \begin{pmatrix}
    \Vec{r_1} & \Vec{r_n} &  \cdots & \Vec{r_3} & \Vec{r_2}
  \end{pmatrix}
  \;=\;
  \begin{pmatrix}
    1      & n      & n-1      & \cdots & 3    & 2      \\
    2      & 1      & n      & \cdots & 4      & 3      \\
    3      & 2      & 1      & \cdots & 5      & 4      \\
    \vdots & \vdots & \vdots & \ddots & \vdots & \vdots \\
    n-1    & n-2    & n-3      & \cdots & 1    & n    \\
    n      & n-1    & n-2      & \cdots &  2   & 1
  \end{pmatrix}.
\end{equation}

If these $n$ rotation vectors are linearly independent, then the dimension of the dihedral permutohedron must be at least $n-1$. This is exactly the condition that the determinant of the matrix above is nonzero.

\begin{definition}
  Let $C = \operatorname{circ}(c_0, c_1, \dots, c_{n-1})$ be the $n \times n$ circulant matrix with first column $(c_0, c_1, \dots, c_{n-1})$. The \emph{associated polynomial} of $C$ is the polynomial whose
  coefficients are the entries of the first column of $C$,
  \[
    \assocpoly{C}(x) = c_0 + c_1 x + c_2 x^2 + \cdots + c_{n-1} x^{n-1}.
  \]
\end{definition}

This associated polynomial allows us to get information about the determinant of our circulant matrix.

\newcommand{\rootunity}{\omega}   

\begin{theorem}
  \cite{circ} Let $C = \operatorname{circ}(c_0, \dots, c_{n-1})$ be a circulant matrix
  with associated polynomial $\assocpoly{C}$, and let
  $\rootunity = e^{2\pi i/n}$. Then
  \[
    \det C = \prod_{k=0}^{n-1} \assocpoly{C}\left(\rootunity^k\right).
  \]
\end{theorem}

With this formula, we can derive one of the two key results for finding the volume of the dihedral permutohedron.

\begin{prop}\label{prop:basis}
The $n$ rotation vertices \eqref{eq:rotations} of the dihedral permutohedron are linearly independent and the dihedral permutohedron has dimension $n - 1$.    
\end{prop}
\begin{proof}
Firstly, just like the permutohedron, $Q_n$ lives in the hyperplane defined by the linear equation $x_1 + x_2 + \cdots + x_n = 1 + 2 + \cdots + n = n(n+1)/2$, so dim $Q_n \leq n-1$. For the circulant matrix defined by our rotation vectors in \eqref{eq:circ}, the associated polynomial is $1 + 2x + 3x^2 + \cdots + nx^{n-1}$. An antiderivative of this polynomial is $x + x^2 + x^3 + \cdots + x^n = x(1 + x + \cdots + x^{n-1}) = x[(1-x^{n})/(1-x)]$, for $x \neq 1$. Thus, the roots of this antiderivative are $0$ and all the $n$th roots of unity besides $1$. Since all these roots are of multiplicity $1$, the derivative, the associated polynomial cannot share any of the same roots. Additionally, it is clear that $1$ is not a root of the associated polynomial. Therefore, the product of the associated polynomial values evaluated at all $n$th roots of unity is nonzero. Therefore, the determinant of this matrix is a product of nonzero complex numbers and is hence nonzero. Thus, the linear space spanned by all $n$ rotation vectors is $n$-dimensional, hence their affine hull has dimension $n-1$. Then, dim $Q_n \geq n-1$ and using the earlier deduction, dim $Q_n = n-1$. \end{proof}

\begin{remark}
    One could choose any basepoint $\rho$ to project the Birkhoff polytope or the dihedral polytope to $\mathbb{R}^n$. For example, one could consider defining the dihedral permutohedron to use the basepoint $(0, 1, \cdots, n-1)$. Note that the above proposition still holds for any choice of basepoint $\rho$ such that the matrix with columns $r_1, r_2, \cdots, r_n$ has nonzero determinant.
\end{remark}

\begin{definition}
    In a polytope $P$, we define the \emph{antipode} of a vertex $v$ of the polytope as a vertex $w$ of the polytope such that $\frac{v+w}{2}$ is equal to the arithmetic mean of all the vertices of $P$; the vertex centroid. We say that $P$ is \emph{centrally symmetric} if every vertex of $P$ has an antipode.
\end{definition}

As we will show next, the dihedral permutohedron is centrally symmetric, with each rotation vertex $r_i$ \eqref{eq:rotations} having an antipodal vertex we denote by $f_i$. As an example, the dihedral permutohedron for $n = 3$ is shown below, with the vertex centroid labeled by $C$ and the antipodal pairs drawn in the same color.

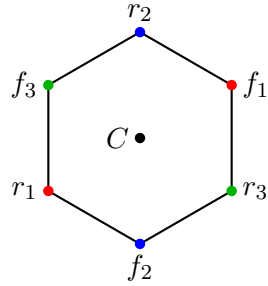
\begin{figure}[h]
\centering
\begin{tikzpicture}[scale=1.4, every node/.style={font=\small}]

  \coordinate (r1) at (210:1);
  \coordinate (f2) at (270:1);
  \coordinate (r3) at (330:1);
  \coordinate (f1) at (30:1);
  \coordinate (r2) at (90:1);
  \coordinate (f3) at (150:1);

  \coordinate (C) at (0,0);

  \draw[thick]
    (r1) -- (f2) -- (r3) -- (f1) -- (r2) -- (f3) -- cycle;

  \fill[red] (r1) circle (1.4pt);
  \node[anchor=east] at (r1) {$r_1$};

  \fill[red] (f1) circle (1.4pt);
  \node[anchor=west] at (f1) {$f_1$};

  \fill[blue] (r2) circle (1.4pt);
  \node[anchor=south] at (r2) {$r_2$};

  \fill[blue] (f2) circle (1.4pt);
  \node[anchor=north] at (f2) {$f_2$};

  \fill[green!70!black] (r3) circle (1.4pt);
  \node[anchor=west] at (r3) {$r_3$};

  \fill[green!70!black] (f3) circle (1.4pt);
  \node[anchor=east] at (f3) {$f_3$};

  \fill[black] (C) circle (1.4pt);
  \node[anchor=east] at (C) {$C$};

\end{tikzpicture}
\end{figure}

\begin{prop}\label{prop:centroid}
    Recalling that $\mathbf{1}$ is the all-ones vector, the dihedral permutohedron is centrally symmetric about the vertex centroid, which is $((n+1)/2)(\mathbf{1})$. Using the rotation vertices in \eqref{eq:rotations}, the vertex centroid is also equal to \[\frac{r_1 + \cdots + r_n}{n}.\]
\end{prop}
\begin{proof} 
    The vertex centroid of the dihedral permutohedron is equal to \[\frac{1}{2n}(\sum_{w\in D_{2n}}w(\rho)) = \frac{2(n(n+1)/2)}{2n}(\mathbf{1}),\]
    as the rotation vector set and reflection vector set have each integer from $1$ to $n$ appear exactly once in each of the indices. This vertex centroid is thus equal to $((n+1)/2)(\mathbf{1})$. The vertex centroid is also evidently equal to $\frac{r_1 + \cdots + r_n}{n}$. Now, consider the reflection $s_2$ \eqref{eq:odd_perm}, \eqref{eq:even_perm} and the vertex $s_2(\rho) = (1,n, n-1, n-2,\cdots, 3,2)$. After cyclically shifting all numbers in this vector to the left we get the reflection vertex $(s_2s_1s_2)(\rho) = (n, n-1, \cdots, 2,1)$. We have that $(r_1 + (s_2s_1s_2)(\rho))/2 = ((n+1)/2)(\mathbf{1})$. Using this equation, we can apply rotations to both sides repeatedly to show that every rotation of the dihedral permutohedron has a distinct reflection as its antipode. Furthermore, if we are considering a rotation $r_i$ of the dihedral permutohedron, this argument shows the antipode of $r_i$ is found by left multiplying $r_i$ by $s_2s_1s_2$.\end{proof}

\begin{remark}
    The above proposition tells us that the vertices of the dihedral permutohedron are determined by the $n$ rotation vertices (or the $n$ reflection vertices). Additionally, all of these are determined by the choice of $\rho$.
\end{remark}

Using the result above, define
\begin{equation}\label{eq:ref}
    f_i := (n+1)(\mathbf{1}) - r_i,\text{ the reflection vertex that is the antipode of $r_i$.}
\end{equation}

\begin{definition}
    The \emph{normalized volume} of a $d$-dimensional lattice polytope $P \subset \mathbb{R}^d$ is $\vol(P):= d!(\Vol(P))$, where the Euclidean volume, $\Vol(P)$, is taken with respect to the affine lattice given by the intersection of $\mathbb{Z}^d$ and the affine span of $P$.
\end{definition}

With this definition in mind, we have proven that our polytope $Q_n$ has dimension $(n-1)$ and we can choose an appropriate lattice. Consider translating $Q_n$ by $-r_1$. $Q_n - r_1$ is a subset of the hyperplane of $\mathbb{R}^n$ defined by $x_1 + x_2 + \cdots + x_{n-1} + x_n = 0$. Hence, we can embed $Q_n - r_1$ into $\mathbb{R}^{n-1}$ by expressing all points of $Q_n - r_1$ in terms of the simple root basis $e_i - e_{i+1}$ for all $1 \leq i \leq n-1$. After this embedding, the dihedral permutohedron is a genuine lattice polytope in $\mathbb{R}^{n-1}$ with respect to the lattice $\mathbb{Z}^{n-1}$, so it makes sense to consider normalized volume.

\begin{remark}
    It is a well-known result that the relative volume of the permutohedron is $n^{n-2}$. However, the normalization convention used for this result is not the same as the convention used in this paper. Hence, the standard formula for the volume of the permutohedron is not directly comparable to the main result of this paper.
\end{remark}

\section{Normalized Volume of the Dihedral Permutohedron}

Computing the normalized volumes of the dihedral permutohedron for small values of $n$ and looking for patterns suggested the following theorem, which will be proved in this section.

\begin{theorem}
    For all $n \geq 3$, the normalized volume of the dihedral permutohedron is 
    \begin{equation}
        \vol(Q_n) = n^{n-2}\binom{n-1}{\lfloor \frac{n-1}{2} \rfloor}.
    \end{equation}
\end{theorem}

The first idea that came to mind in proving this formula was looking for a way to triangulate the dihedral permutohedron for all $n \geq 3$. Seeing patterns in the possible triangulations for small cases suggested that a particular type of triangulation existed for all $n \geq 3$.

\begin{definition} \cite{stanley80}
    Given a convex polytope $P$ and a subdivision $S = \{S_1, \cdots, S_n \}$ of $P$, \emph{pulling} at a vertex $v \in S_1 \cup \cdots \cup S_n$ yields the following refined subdivision:
    \[
    T = \{\, S_i : v \notin S_i \,\} \;\cup\; \{\, \conv(F \cup \{v\}) \mid \ v \in S_i,\ F \text{ a facet of } S_i,\ v \notin F \,\}.
    \]
\end{definition}

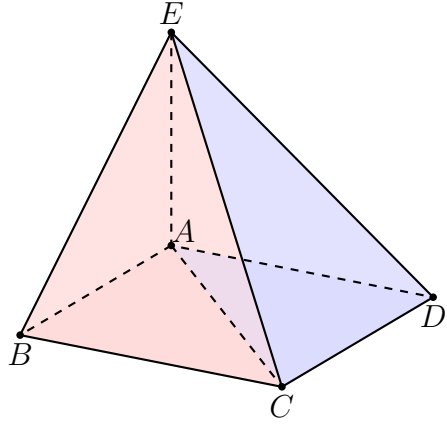
\begin{figure}[H]
\centering
\tdplotsetmaincoords{70}{120}
\begin{tikzpicture}[
  tdplot_main_coords,
  line cap=butt,
  line join=round,
  c/.style={circle,fill,inner sep=1pt},
  declare function={a=4;h=3;}
]

  \path
    (0,0,0) coordinate (A)
    (a,0,0) coordinate (B)
    (a,a,0) coordinate (C)
    (0,a,0) coordinate (D)
    (0,0,h) coordinate (E);

  \fill[red!25,opacity=0.45]
    (E) -- (B) -- (C) -- cycle;

  \fill[red!18,opacity=0.35]
    (A) -- (B) -- (C) -- cycle;

  \fill[blue!25,opacity=0.45]
    (E) -- (C) -- (D) -- cycle;

  \fill[blue!18,opacity=0.35]
    (A) -- (C) -- (D) -- cycle;

  \draw[thick]
    (E) -- (D) -- (C) -- (B) -- cycle
    (E) -- (C);

  \draw[thick,dashed]
    (E) -- (A) -- (D)
    (A) -- (B);

  \draw[thick,dashed]
    (A) -- (C);

  \path foreach \p/\g in {
    A/50,
    B/-90,
    C/-90,
    D/-90,
    E/90
  }{
    (\p) node[c]{}
    +(\g:2.5mm) node {$\p$}
  };

\end{tikzpicture}

\caption{As an example, we can take this square pyramid, start from the trivial subdivision and pull at the vertex $A$ to subdivide the square pyramid into the tetrahedra ABCE and ACDE.}
\end{figure}

\begin{theorem}
\label{thm:4.2}
    For odd $n \geq 3$, starting with the trivial subdivision and pulling at the rotation vertex $r_1$ of \eqref{eq:rotations} yields a triangulation with exactly $\binom{n-1}{(n-1)/2}$ simplices of normalized volume $n^{n-3}$ and $\frac{n-1}{2}\binom{n-1}{(n-1)/2}$ simplices of normalized volume $2n^{n-3}$. Thus, the normalized volume of the odd dihedral permutohedron is $n^{n-2}\binom{n-1}{(n-1)/2}$. 
\end{theorem}

\begin{theorem}
\label{thm:4.3}
    For even $n > 3$, pulling at $r_1$ and then $f_1$ \eqref{eq:ref} yields a triangulation with $\frac{n}{2}\binom{n-1}{\lfloor(n-1)/2 \rfloor}$ simplices of normalized volume $2n^{n-3}$. Thus, the normalized volume of the even dihedral permutohedron is $n^{n-2}\binom{n-1}{\lfloor(n-1)/2 \rfloor}$.
\end{theorem}

We will now prove these nice triangulations exist for the dihedral permutohedron by getting a description of the facets of the dihedral permutohedron. The proof of this result will proceed differently for odd $n$ and even $n$.

\begin{lemma}[Affine Independence Criteria For Odd $n$]
\label{lem:4.4}
For odd $n$, for any set of vertices of the dihedral permutohedron, the set is linearly independent and hence affinely independent if and only if it has at most $n$ vertices and has fewer than two pairs of vertices that are related by central symmetry.

\end{lemma}
    
\begin{proof}
    If the set of vectors contains two pairs of antipodal vertices, for example, $r_1, f_1$ and $r_2, f_2$, then by \eqref{eq:ref} we have the equation $(r_1 + f_1)/2 - (r_2+f_2)/2 = 0$, showing this set is affinely dependent. This shows one direction of the if and only if. To prove the other direction, consider the vectors $r_1, r_2, \cdots, r_{n}$, which we know to be linearly independent by our previous result. Let $m := \frac{r_1 + \cdots + r_n}{n}$. By Proposition \ref{prop:centroid}, $m = \frac{n+1}{2}\mathbf{1}$, so \eqref{eq:ref} gives
    \[f_i = (n+1)\mathbf{1} - r_i = 2m - r_i.\]
    
    Suppose $c_1f_1 + c_2f_2 + \cdots + c_if_i + c_{i+1}r_{i+1} + \cdots + c_{n}r_{n} = 0$. Substituting $f_k = 2m - r_k$ from \eqref{eq:ref}, we have
    \[
    c_1(2m - r_1) + \cdots + c_i(2m - r_i) + c_{i+1}r_{i+1} + \cdots + c_{n}r_{n} = 0.
    \]
    The left-hand side is equal to $-c_1r_1 + \cdots -c_i(r_i) + c_{i+1}r_{i+1} + \cdots + c_{n}r_{n} + 2(c_1 + \cdots + c_i)m$. Now, we can distribute the coefficients of $m$ to see that this is equal to
    \[
    (2/n(c_1 + \cdots + c_i) - c_1)r_1 + \cdots + (2/n(c_1 + \cdots + c_i)-c_i)r_i + (2/n(c_1 + \cdots + c_i) 
    \]
    
    \[
     + c_{i+1})r_{i+1} + \cdots + (2/n(c_1 + \cdots + c_i)+ c_n)r_n.
    \]
    Since these vectors $r_1, \ldots, r_n$ are linearly independent, we have that each of these coefficients is zero. Namely, if we let $x := c_1 + \cdots + c_i$, then $c_1 = c_2 = \cdots = c_i$ and thus we have that $x = i(2/n)(x) = (2i/n)(x)$. This equation for $x \neq 0$ is true if and only if $i = (n/2)$, which cannot be true for any integer $i$ for odd $n$. Thus, $x = 0$ and all the coefficients are $0$, so $f_1, f_2, \cdots f_i, r_{i+1}, \cdots, r_n$ are linearly independent and hence affinely independent. Notice that our proof did not depend on the specific indices of the rotations and reflections, only the existence of $n-i$ rotations and $i$ reflections in a list. Hence, any set of $n$ vertices of the dihedral permutohedron with no antipodal pairs is linearly independent. Taking subsets, any set of vertices with no antipodal pairs is linearly independent. Now, suppose we have $n$ vertices with one antipodal pair. Fix a permutation $\sigma \in S_n$ and an integer $1 \le j \le n-1$, and consider the vertex list \[
r_{\sigma(1)}, \cdots, r_{\sigma(j)}, \; f_{\sigma(1)}, f_{\sigma(j+1)}, \cdots, f_{\sigma(n-1)},\] consisting of the $j$ rotations indexed by $\sigma(1), \ldots, \sigma(j)$, together with
$n-j$ reflections: the reflection $f_{\sigma(1)}$ antipodal to $r_{\sigma(1)}$, and the
$n-1-j$ further reflections indexed by $\sigma(j+1), \ldots, \sigma(n-1)$ (all distinct
from $\sigma(1), \ldots, \sigma(j)$, since $\sigma$ is a permutation). This list has
exactly one antipodal pair, namely $\{r_{\sigma(1)}, f_{\sigma(1)}\}$. Suppose that:
\[
c_1 r_{\sigma(1)} + \cdots + c_j r_{\sigma(j)} + c_{j+1} f_{\sigma(1)} + c_{j+2} f_{\sigma(j+1)} + \cdots + c_n f_{\sigma(n-1)} = 0.
\]

Substituting $f_{\sigma(1)} = 2m - r_{\sigma(1)}$ and $f_{\sigma(k)} = 2m - r_{\sigma(k)}$
for $k = j+1, \ldots, n-1$ (by \eqref{eq:ref}), the left-hand side becomes:
\[
c_1 r_{\sigma(1)} + \cdots + c_j r_{\sigma(j)} + c_{j+1}(2m - r_{\sigma(1)}) + c_{j+2}(2m - r_{\sigma(j+1)}) + \cdots + c_n(2m - r_{\sigma(n-1)}),
\]
which is equal to
\[
(c_1 - c_{j+1}) r_{\sigma(1)} + \cdots + c_j r_{\sigma(j)} - c_{j+2} r_{\sigma(j+1)} - \cdots - c_n r_{\sigma(n-1)} + 2(c_{j+1} + \cdots + c_n)m.
\]
Writing $m = \frac{1}{n}\sum_{k=1}^n r_k$ and distributing the coefficient of $m$
across all $n$ rotation vectors $r_1, \ldots, r_n$, this becomes a linear combination
of $r_1, \ldots, r_n$ in which the coefficient of $r_{\sigma(1)}$ is
$c_1 - c_{j+1} + \frac{2}{n}(c_{j+1} + \cdots + c_n)$, the coefficient of $r_{\sigma(k)}$
for $2 \le k \le j$ is $c_k + \frac{2}{n}(c_{j+1} + \cdots + c_n)$, the coefficient of
$r_{\sigma(k)}$ for $j+1 \le k \le n-1$ is $-c_{k+1} + \frac{2}{n}(c_{j+1} + \cdots + c_n)$,
and the coefficient of every remaining $r_l$ (those with $l \notin \{\sigma(1), \ldots, \sigma(n-1)\}$)
is $\frac{2}{n}(c_{j+1} + \cdots + c_n)$. Since $r_1, \ldots, r_n$ are linearly independent, every one of these coefficients
must vanish. In particular, letting \[y := c_{j+1} + \cdots + c_n,\] the coefficient of
any remaining $r_l$ gives $\frac{2}{n} y = 0$, so $y = 0$. Then the coefficient of
$r_{\sigma(k)}$ for $2 \le k \le j$ gives $c_k = 0$; the coefficient of $r_{\sigma(k)}$ for
$j+1 \le k \le n-1$ gives $c_{k+1} = 0$, i.e. $c_{j+2} = \cdots = c_n = 0$; and since
$y = c_{j+1} + \cdots + c_n = 0$ with $c_{j+2} = \cdots = c_n = 0$, we get $c_{j+1} = 0$
as well. Finally, the coefficient of $r_{\sigma(1)}$ gives $c_1 - c_{j+1} = 0$, so
$c_1 = 0$. Thus every coefficient vanishes, and this one-antipodal-pair set is linearly independent. Once again, taking subsets of these $n$-element sets shows that any set of vertices with one antipodal pair is linearly independent.\end{proof}

\begin{cor}
    The odd dihedral permutohedron is simplicial.
\end{cor}

\begin{proof}
Consider some proper face of the dihedral permutohedron. If this face contains an antipodal pair $r_i, f_i$, then this face contains, by \eqref{eq:ref}, $\frac{r_i + f_i}{2} = \frac{(n+1)}{2}\mathbf{1}$, the vertex centroid of the dihedral permutohedron. However, faces are contained in the boundary so they certainly cannot contain the vertex centroid. Thus, this face has no antipodal pairs and by lemma \ref{lem:4.4} the face is a simplex. 
\end{proof}

\begin{lemma}[Affine Independence Criteria For Even $n$] \label{lem:4.5}
For even $n$, given a set of $<n$ vertices of the dihedral permutohedron, the set is affinely independent if and only if there are fewer than two antipodal pairs. For a set of $n$ vertices, the set is affinely independent if and only if it either has exactly one antipodal pair, or has no antipodal pairs and unequal numbers of rotation and reflection vertices.
\end{lemma}
    
\begin{proof}
    The fact that two antipodal pairs yield affine dependence follows from the same argument as the odd case. Recall the proof from the previous lemma \ref{lem:4.4} for the case of no antipodal pairs with $n$ vertices. For even $n$, setting $i$ = $n/2$ in the previous proof yields a solution, which is the case that half of our vectors are rotations and half are reflections. Thus, it is possible to get a solution to that equation for even $n$. In more detail, fix $\sigma \in S_n$ and select $n/2$ rotations and the corresponding $n/2$ reflections that yield no antipodal pairs \[
    r_{\sigma(1)}, \cdots r_{\sigma(n/2)}, f_{\sigma(n/2 + 1)}, \cdots f_{\sigma(n)}.
    \] Using the definition of the $f_i$ in \eqref{eq:ref}, we have that $r_{\sigma(1)} +  \cdots + r_{\sigma(n/2)} - f_{\sigma(n/2 + 1)} - \cdots - f_{\sigma(n)} = r_{\sigma(1)} +  \cdots + r_{\sigma(n/2)} + r_{\sigma(n/2 + 1)} + \cdots + r_{\sigma(n)} - nm = 0$. Since the coefficients in this linear combination sum to $0$, these vectors are affinely dependent. For the one antipodal pair case for any set, we can refer to the same proof as the odd case, as that argument had no dependence on parity. For sets of vertices with no antipodal pairs and $<n$ elements, we can adjoin vertices until we get a set with one antipodal pair and $n$ vertices. This set is linearly independent by the previous case, so the original set also had to be linearly independent.\end{proof}

Now that we have criteria for when a set of vertices of the dihedral permutohedron are affinely independent, we can classify the facets of the dihedral permutohedron for the even and odd cases.

\begin{cor}
    The even dihedral permutohedron is almost simplicial; all proper faces are simplices except for the facets.
\end{cor}

\begin{proof}
    Firstly, any set of $>n$ vertices must contain an antipodal pair, so none of these sets can form a facet. For any set of $<n$ vertices whose convex hull forms a face, there must be no antipodal pairs, otherwise the face contains the vertex centroid, a contradiction. Thus, lemma \ref{lem:4.5} shows that the convex hull is a simplex, as a face cannot contain an antipodal pair. As we will prove later in this paper, if a set of $n$ vertices forms a face of the dihedral permutohedron, this face must be a facet. Hence, any face that isn't a facet must have $<n$ vertices and is therefore a simplex.
\end{proof}

The fact that the even dihedral permutohedron is not quite simplicial will introduce one additional layer of difficulty in proving the main theorem of this paper for even $n$ later. 

\begin{definition}
\label{def:linear_functional}
    Given some $\sigma \in S_n$, since $r_1, \cdots, r_n$ form a basis of $\mathbb{R}^n$ by proposition \ref{prop:basis}, we can define $\varphi_\sigma$ to be the unique linear functional that vanishes on $r_{\sigma(1)}, r_{\sigma(2)}, \cdots ,r_{\sigma(n/2)}$ and is equal to $1$ when evaluated at the other $\frac{n}{2}$ rotation vertices.
\end{definition}

\begin{lemma}
\label{lem:4.6}
    For even $n > 3$, \begin{equation}\conv(
    r_{\sigma(1)}, r_{\sigma(2)}, \cdots ,r_{\sigma(n/2)}, f_{\sigma(n/2+1)}, \cdots, f_{\sigma(n)})
    \end{equation}
    is a facet for any $\sigma \in S_n$.
\end{lemma}
\begin{proof}
    Consider the linear functional $\varphi_\sigma$ (\ref{def:linear_functional}). Then, by \eqref{eq:ref},\[
    \varphi_\sigma(f_{\sigma(j)}) = \varphi_\sigma((n+1)(\mathbf{1}) - r_{\sigma(j)})\]
    Note that $(n+1)(\mathbf{1}) = 2((n+1)/2)(\mathbf{1}) = 2(r_1 + \cdots + r_{n})/n$. Thus, $\varphi_\sigma((n+1)(\mathbf{1})) = (2/n(n/2))(1) = 1$. This means $\varphi_\sigma(f_{\sigma(j)}) = 0$ for $j > n/2$ and $\varphi_\sigma(f_{\sigma(j)}) = 1$ for $j \leq n/2$. Indeed, this linear functional vanishes on exactly the $n$ desired vertices. Furthermore, the linear inequality $\varphi_\sigma(v) \geq 0$ is true for all vertices of the dihedral permutohedron, since it was constructed to be true for the rotations and is proven to be true for the reflections by the computation above. This inequality is hence a valid inequality of the dihedral permutohedron and these vertices form a face of the dihedral permutohedron. By the affine dependence criteria, if we remove one of the $n$ vertices of our face and take the convex hull, we get a simplex of dimension $n-2$. Since this face is not the entire dihedral permutohedron, the only possibility is that these $n$ vertices form a facet. \end{proof}

\begin{lemma}
\label{lem:4.7}
    For odd $n \geq 3$, \begin{equation} \label{eq:facets}
    F_\sigma = \conv(r_{\sigma(1)}, r_{\sigma(2)}, \cdots ,r_{\sigma((n-1)/2)}, f_{\sigma((n+1)/2)}, \cdots, f_{\sigma(n-1)}) \end{equation}
    is a facet for any $\sigma \in S_n$.
\end{lemma}
\begin{proof}
 Since the rotation vectors form a basis of $\mathbb{R}^n$, we can choose a unique linear functional $\psi_\sigma$ that vanishes when evaluated at the first $(n-1)/2$ rotation vectors, is equal to $1$ when evaluated at the next $(n-1)/2$ rotation vectors, and is equal to $1/2$ when evaluated at $r_{\sigma(n)}$. Then, by \eqref{eq:ref},\[
    \psi_\sigma(f_{\sigma(j)}) = \psi_\sigma((n+1)(\mathbf{1}) - r_{\sigma(j)}).
    \] Note that $(n+1)(\mathbf{1}) = 2((n+1)/2)(\mathbf{1}) = 2(r_1 + \cdots + r_{n})/n$. Thus, $\psi_\sigma((n+1)(\mathbf{1})) = 2/n((n-1)/2 + 1/2) = 2/n(n/2) = 1$. This means $\psi_\sigma(f_{\sigma(n)}) = 1/2$, $\psi_\sigma(f_{\sigma(j)}) = 0$ for $n > j > (n-1)/2$ and $\psi_\sigma(f_{\sigma(j)}) = 1$ for $j \leq (n-1)/2$. Indeed, this linear functional vanishes on exactly the $n-1$ desired vertices. Furthermore, the linear inequality $\psi_\sigma(v) \geq 0$ is true for all vertices of the dihedral permutohedron, since it was constructed to be true for the rotations and is proven to be true for the reflections by the computation above. This inequality is hence a valid inequality of the dihedral permutohedron and these vertices form a face of the dihedral permutohedron. By lemma \ref{lem:4.4}, any set of vertices with no antipodal pairs is affinely independent for odd $n$, so these vertices have an affine hull that is a simplex of dimension $n-2$. Indeed, these vertices form a facet. \end{proof}

We have now shown that certain subsets of the vertices form facets for the odd and even cases. We will now show that all the facets of the dihedral permutohedron are of the forms above using a centroid argument, as has become common in this paper.

\begin{lemma}
\label{lem:4.8}
    For odd $n$, a facet must have $n-1$ vertices with $(n-1)/2$ rotation vertices, $(n-1)/2$ reflection vertices \eqref{eq:ref}, and no antipodal pairs, exactly the facets described in the previous lemma.
\end{lemma}
\begin{proof}
    Consider some facet of the odd dihedral permutohedron. Since the odd dihedral permutohedron is simplicial, the facet must contain $(n-1)$ vertices. Furthermore the facet can contain no antipodal pairs, as the vertex centroid, an interior point, is a convex combination of any two antipodal points. Thus, this facet must be a simplex of dimension $n-2$ defined by some vertices $r_{\sigma(1)}, r_{\sigma(2)}, \cdots, r_{\sigma(i)}, f_{\sigma(i+1)}, \cdots, f_{\sigma(n-1)}$. Without loss of generality, we can assume half or more of these vertices are rotations, as if this is not the case, we can apply the central symmetry map of the dihedral permutohedron to get a facet composed of half rotations or more. In other words, $i \geq (n-1)/2$. Consider the centroid of this facet\[
    (r_{\sigma(1)} + r_{\sigma(2)} + \cdots +r_{\sigma(i)} + \cdots + f_{\sigma(n-1)})/(n-1).
    \] This expression is equal to $(1/(n-1))[(r_{\sigma(1)} + \cdots +  r_{\sigma(n)}) - r_{\sigma(i+1)} - \cdots - r_{\sigma(n)} + \cdots + f_{\sigma(n-1)}] = (1/(n-1))[(n(n+1)/2)(\mathbf{1}) - r_{\sigma(i+1)} - \cdots - r_{\sigma(n)} + \cdots +f_{\sigma(n-1)}]$. Distributing the factors of the constant vector yields this expression being equal to $(1/(n-1))[((n+1)(\mathbf{1}) - r_{\sigma(i+1)}) + \cdots + ((n + 1)(\mathbf{1}) - r_{\sigma(n)}) + \cdots +f_{\sigma(n-1)} + ((n(n+1)/2) - (n+1)(n-i))(\mathbf{1})]$. Using \eqref{eq:ref}, we have the centroid being equal to \[
    (1/(n-1))[2(f_{\sigma(i+1)} + \cdots f_{\sigma(n-1)}) + f_{\sigma(n)} + (n/2 - (n-i))(n+1)(\mathbf{1})].
    \] This expression is equal to $(1/(n-1))[2(f_{\sigma(i+1)} + \cdots f_{\sigma(n-1)}) + f_{\sigma(n)} + (1/n)(i-n/2)\sum_{k=1}^n(r_{\sigma(k)} + f_{\sigma(k)})]$. We want to check if this is a convex combination. The first $n-i$ coefficients inside the brackets sum to $2(n-i) - 1 = 2n - 2i - 1$. The second set of coefficients sum to $2(i-n/2) = 2i - n$. Adding the two together yields $n - 1$, which after dividing by all $n-1$ yields a linear combination with coefficients that sum to $1$. The last coefficient, $(1/n)(i - n/2)$ is positive for integer $i > (n-1)/2$. Thus, for $i > (n-1)/2$, the centroid can be expressed as a positive convex combination of all the vertices of the dihedral permutohedron, and hence lies in the relative interior of the dihedral permutohedron. Therefore, these vertices cannot form a facet unless $i = (n-1)/2$ as facets need to be contained in the boundary of polytopes. Indeed, the lemma is proven. \end{proof}

\begin{cor}
    The facets of the odd dihedral permutohedron are formed exactly by choosing $(n-1)/2$ rotation vertices and $(n-1)/2$ reflection vertices to be included in the facet, with no antipodal pairs. This yields $n\binom{n-1}{(n-1)/2}$ facets for the odd dihedral permutohedron.
\end{cor}
\begin{proof}
    By lemma \ref{lem:4.7} we have that all simplices formed by taking $(n-1)/2$ rotation vertices and $(n-1)/2$ reflection vertices form facets. From lemma \ref{lem:4.8}, we have that these are the only simplices that could be facets. To construct such a facet, you have $n$ choices of which antipodal pair to exclude, and then there are $\binom{n-1}{(n-1)/2}$ choices for the rotations after making this choice. This yields $n \binom{n-1}{(n-1)/2}$ total facets.
\end{proof}

\begin{lemma}
\label{lem:4.9}
    For even $n > 3$, a facet must have $n$ vertices with an equal number of rotation and reflection vertices, and no antipodal pairs. These are exactly the facets described in lemma \ref{lem:4.6}.
\end{lemma}

\begin{proof}
    Since a facet cannot contain an antipodal pair, for even $n > 3$, from our lemma on affine independence, we have that a facet of the dihedral permutohedron must either be a simplex with $(n-1)$ vertices, or be formed by $n$ vertices with an equal number of rotations and reflections. For the simplex case, we can employ the same centroid argument as for the odd case above, but with $(n/2)$ rotations or more, which corresponds to $i \geq n/2$. Looking at the equations in the previous proof, this eliminates the possibility of a facet of this form having more than $n/2$ rotations or less than $n/2 - 1$ rotations.

    Thus, a facet could have $n$ vertices with $n/2$ rotations and $n/2$ reflections, could be a simplex with $n/2$ rotations and $n/2 - 1$ reflections, or could be a simplex with $n/2 - 1$ rotations and $n/2$ reflections.  By lemma \ref{lem:4.6} any simplex with $n/2$ rotations and $n/2 - 1$ reflections is properly contained in a facet with $n$ vertices and has the same dimension, and is hence not a facet. The same holds for simplices with $n/2 - 1$ rotations and $n/2$ reflections. Hence, a facet must have $n$ vertices with equal rotations and reflections. \end{proof}

\begin{cor}
    The facets of the even dihedral permutohedron are formed exactly by choosing $n/2$ rotations and $n/2$ reflections, with no antipodal pairs. This yields exactly $\binom{n}{n/2}$ facets.
\end{cor}

\begin{proof}
    For even $n$, lemma \ref{lem:4.6} shows that any choice of such vertices forms a facet. Lemma \ref{lem:4.9} shows that any facet must be of this form. Since there are $n$ antipodal pairs and the facet is determined by choosing the $n/2$ rotation vertices, there are $\binom{n}{n/2}$ facets. 
\end{proof}

Let us now perform the pulling triangulation for odd $n \geq 3$ described by theorem \ref{thm:4.2}. We start with the trivial subdivision $\{Q_n\}$. After pulling at $r_1$, this becomes the subdivision $\{\conv(F \cup r_1) \mid r_1 \notin F\}$, where $F$ is a facet of the dihedral permutohedron. There are two cases for a facet not containing $r_1$. Before discussing these two cases, there is a useful normalized volume computation we can perform.

\begin{lemma}
\label{lem:4.10}
    For odd $n$, the normalized volume of the simplex with vertex set $\{r_i\}_{i=1}^n$ \eqref{eq:rotations} is $n^{n-2}$.
\end{lemma}
\begin{proof}
    To do this computation, we can embed $Q_n$ in $(n-1)$-dimensional space as a lattice polytope to compute the normalized volume as the determinant of an $(n-1) \times (n-1)$ matrix. We will proceed in the same way as before. First, translate all the vertices in this list by $-r_1$. We now have the vertices $\Vec{0}, r_2 - r_1, \cdots, \text{and} \ r_n - r_1$. We want to compute the normalized volume after embedding the simplex these vertices form into an $(n-1)$-dimensional space. Since the vertices of the dihedral permutohedron live in the hyperplane $\sum x_i = n(n+1)/2$, after this translation, the vertices live in the hyperplane $\sum x_i = 0$. We now have a natural way of embedding into $(n-1)$-dimensional space by expressing all these translated vertices in terms of the simple root basis of $S_n$. Let us see what the images of these translated vertices are after embedding them in $(n-1)$-dimensional space. We have \[r_i - r_1 = (1 + (i-1), \cdots, n + (i-1)) \ \text{mod } n - (1, \cdots, n). \] This expression is equal to $(i, i+1, \cdots, n, 1, \cdots, i-1) - (1,\cdots , n) = (i-1, i-1, \cdots, n-(n-i + 1), 1 - (n-i+2), \cdots, (i-1) - n) = (i-1, i-1, \cdots, i-1, i-n-1, i-n-1, \cdots, i-n-1)$. Let $\Bar{r}_i$ be our embedding of the vertex $r_i - r_1$ in $(n-1)$-dimensional space. Since embedding a vector in the simple root basis corresponds to taking partial sums of a vector, we have \[\Bar{r}_i = ((i-1), 2(i-1), \cdots, (n-i+1)(i-1), (n-i+2)(i-1) - n, \cdots, n -(i-1)) \] for $i > 1$, with $\Bar{r_1} = 0$. Now that we have 
    have embedded the dihedral permutohedron in $\mathbb{R}^{n-1}$ as a lattice polytope, and know the images of the rotation vertices under this embedding, the normalized volume of the simplex $\{r_i\}_{i=1}^n$ is equal to the absolute value of the determinant of the $(n-1) \times (n-1)$ matrix with columns $\{\Bar{r}_i\}_{i=2}^{n}$. This matrix is equal to \[\begin{pmatrix}
    1     & 2      & 3      & \cdots & n-2    & n-1      \\
    2      & 4      & 6      & \cdots & 2(n-2)  & 2(n-1)-n      \\
    3      & 6      & 9      & \cdots & 3(n-2) -n  & 3(n-1)-2n      \\
    \vdots & \vdots & \vdots & \ddots & \vdots & \vdots \\
    n-2    & 2(n-2)    & 3(n-2) - n & \cdots & 4    & 2    \\
    n-1      & 2(n-1) - n    & 3(n-1)-2n     & \cdots &  2   & 1
  \end{pmatrix}.\] Denote the columns of this matrix by $c_1, \cdots, c_{n-1}$. Subtract $2c_1$ from $c_2$, subtract $3c_1$ from $c_3$, and so on. After these operations we have the matrix:

\[\begin{pmatrix}
    1  & 0      & 0      & \cdots & 0    & 0     \\
    2  & 0      & 0      & \cdots & 0  & -n      \\
    3  & 0      & 0      & \cdots & -n  & \vdots      \\
    \vdots & \vdots & \vdots & \ddots & \vdots & \vdots \\
    n-2    & 0    & -n & \cdots & \cdots   & \cdots   \\
    n-1      & -n    & \cdots      & \cdots &  \cdots   & \cdots
\end{pmatrix}.\] Swapping columns, we can get the lower-triangular matrix
\[\begin{pmatrix}
    1  & 0      & 0      & \cdots & 0    & 0     \\
    2  & -n      & 0      & \cdots & 0  &  0    \\
    3  & \vdots      & -n     & \cdots & 0  & 0     \\
    \vdots & \vdots & \vdots & \ddots & \vdots & \vdots \\
    n-2    & \cdots     & \cdots & \cdots & -n   & 0 \\
    n-1      &  \cdots    & \cdots      & \cdots &  \cdots   & -n
  \end{pmatrix}.\] This matrix has the absolute value of its determinant equal to $n^{n-2}$. Since none of the column operations performed change the absolute value of the determinant, the normalized volume of our desired simplex is indeed $n^{n-2}$. \end{proof}

We can now discuss the two cases of facets not containing $r_1$ and compute the volumes of the corresponding simplices after adjoining $r_1$.

\begin{lemma}
\label{lem:4.11}
   For odd $n$, if a facet of the dihedral permutohedron doesn't contain $r_1$ or $f_1$, then after adjoining $r_1$, the vertices form a simplex of dimension $n-1$ with normalized volume $n^{n-3}$.
\end{lemma}
\begin{proof}
    Such a facet has vertices $r_{\sigma(2)}, \ldots, r_{\sigma((n+1)/2)}, f_{\sigma((n+3)/2)}, \ldots, f_{\sigma(n)}$, where $\sigma(1) = 1$. Since the odd dihedral permutohedron is simplicial, and adjoining $r_1$ yields no antipodal pairs, our result is an $(n-1)$-dimensional simplex with vertices
    \[
    r_{\sigma(1)}, r_{\sigma(2)}, \ldots, r_{\sigma((n+1)/2)}, \; f_{\sigma((n+3)/2)}, \ldots, f_{\sigma(n)},
    \] where $\sigma(1) = 1$. For notational shorthand define the vertex centroid

    \begin{equation}\label{eq:centroid}
        m := \frac{r_1 + r_2 + \cdots + r_n}{n} = (\frac{n+1}{2})\,\mathbf{1}.
    \end{equation}
    
    By \eqref{eq:ref}, $f_i = (n+1)\mathbf{1} - r_i = 2m - r_i$. For vectors $\Vec{v}_1, \ldots, \Vec{v}_{n-1}$, write \[\bigwedge_{i=1}^{n-1}\Vec{v_i} = \Vec{v}_1 \wedge \cdots \wedge \Vec{v}_{n-1}\] for the determinant of the matrix with columns $\Vec{v}_1, \ldots, \Vec{v}_{n-1}$, where each vector is embedded into $\mathbb{R}^{n-1}$ via the simple root basis. Thus, the normalized volume of the simplex defined by these vertices is given by
    \[ D =
      \left|\;
      \bigwedge_{k=2}^{(n+1)/2} \bigl(\Vec{r}_{\sigma(k)} - \Vec{r}_{\sigma(1)}\bigr)
      \;\wedge\;
      \bigwedge_{j=(n+3)/2}^{n} \bigl(2\Vec{m} - \Vec{r}_{\sigma(j)} - \Vec{r}_{\sigma(1)}\bigr)
      \;\right|.
    \] Subtract the last factor from all the other factors containing a summand of $2\Vec{m}$. We get
    \[ D =
      \left|\;
      \bigwedge_{k=2}^{(n+1)/2} \bigl(\Vec{r}_{\sigma(k)} - \Vec{r}_{\sigma(1)}\bigr)
      \;\wedge\;
      \bigwedge_{j=(n+3)/2}^{n-1} \bigl(-\Vec{r}_{\sigma(j)} + \Vec{r}_{\sigma(n)}\bigr)
      \;\wedge\;
      \bigl(2\Vec{m} - \Vec{r}_{\sigma(n)} - \Vec{r}_{\sigma(1)}\bigr)
      \;\right|.
    \]
    Since changing signs doesn't affect absolute value, we can derive
    \[ D =
      \left|\;
      \bigwedge_{k=2}^{(n+1)/2} \bigl(\Vec{r}_{\sigma(k)} - \Vec{r}_{\sigma(1)}\bigr)
      \;\wedge\;
      \bigwedge_{j=(n+3)/2}^{n-1} \bigl(\Vec{r}_{\sigma(j)} - \Vec{r}_{\sigma(n)}\bigr)
      \;\wedge\;
      \bigl(2\Vec{m} - \Vec{r}_{\sigma(n)} - \Vec{r}_{\sigma(1)}\bigr)
      \;\right|,
    \]
    where we flipped the signs of the factors corresponding to reflection vertices. Now, we can try to eliminate the factor of $2\Vec{m}$ in the last factor. We can do this by subtracting $\tfrac{2}{n}$ times each of the other factors from the last factor. After these operations, we have the following expression as the last factor:
    \[
      \bigl(2\Vec{m} - \Vec{r}_{\sigma(n)} - \Vec{r}_{\sigma(1)}\bigr)
      \;+\; \frac{n-1}{n}\,\Vec{r}_{\sigma(1)}
      \;+\; \frac{n-3}{n}\,\Vec{r}_{\sigma(n)}
      \;-\; \frac{2}{n}\bigl(\Vec{r}_{\sigma(2)} + \cdots + \Vec{r}_{\sigma(n-1)}\bigr).
    \]
    This expression is equal to:
    \[
      \frac{2}{n}\bigl(\Vec{r}_{\sigma(1)} + \Vec{r}_{\sigma(n)}\bigr)
      \;-\; \Vec{r}_{\sigma(n)} - \Vec{r}_{\sigma(1)}
      \;+\; \frac{n-1}{n}\,\Vec{r}_{\sigma(1)}
      \;+\; \frac{n-3}{n}\,\Vec{r}_{\sigma(n)},
    \]
    which is equal to
    \[
      \frac{1}{n}\,\Vec{r}_{\sigma(1)} \;-\; \frac{1}{n}\,\Vec{r}_{\sigma(n)}.
    \] Thus, the determinant $D$ is equal to
    \[ D =
      \frac{1}{n}
      \left|\;
      \bigwedge_{k=2}^{(n+1)/2} \bigl(\Vec{r}_{\sigma(k)} - \Vec{r}_{\sigma(1)}\bigr)
      \;\wedge\;
      \bigwedge_{j=(n+3)/2}^{n-1} \bigl(\Vec{r}_{\sigma(j)} - \Vec{r}_{\sigma(n)}\bigr)
      \;\wedge\;
      \bigl(\Vec{r}_{\sigma(n)} - \Vec{r}_{\sigma(1)}\bigr)
      \;\right|.
    \]
    After adding the last factor to all factors that contain a summand of $-\Vec{r}_{\sigma(n)}$, we get
    \[
      \frac{1}{n}
      \left|\;
      \bigwedge_{k=2}^{n} \bigl(\Vec{r}_{\sigma(k)} - \Vec{r}_{\sigma(1)}\bigr)
      \;\right|
      \;=\; n^{n-3},
    \]
    by the previous lemma \ref{lem:4.10}. Indeed, the normalized volume is as described. \end{proof}
    
\begin{lemma}
\label{lem:4.12}
    For odd $n$, if a facet contains $f_1$, but not $r_1$, then the normalized volume of the $(n-1)$-dimensional simplex formed by adjoining $r_1$ is $2n^{n-3}$.
\end{lemma}

\begin{proof}
    After adjoining $r_1$ to the given facet, by the affine dependence criteria, the result is an $(n-1)$-dimensional simplex equal to \[ \conv(F_\sigma \cup f_{\sigma(1)}),\]
    where $F_\sigma$ is defined in \eqref{eq:facets} and $\sigma(1) = 1$. Thus, using the determinant notation from lemma \ref{lem:4.11} and the relation $f_i = 2m - r_i$ from \eqref{eq:ref}, where $m$ \eqref{eq:centroid} is the vertex centroid, we are computing the absolute value of the determinant
    \[\Bigl|\;
      \bigwedge_{k=2}^{(n+1)/2} \bigl(\Vec{r}_{\sigma(k)} - \Vec{r}_{\sigma(1)}\bigr)
      \;\wedge\;
      \bigwedge_{j=(n+3)/2}^{n-1} \bigl(2\Vec{m} - \Vec{r}_{\sigma(j)} - \Vec{r}_{\sigma(1)}\bigr) \\\wedge\;
      \bigl(2\Vec{m} - 2\Vec{r}_{\sigma(1)}\bigr)
      \;\Bigr|\]
    Subtract the last factor from all the other factors containing a summand of $2\Vec{m}$. We get:
    \[\Bigl|\;
      \bigwedge_{k=2}^{(n+1)/2} \bigl(\Vec{r}_{\sigma(k)} - \Vec{r}_{\sigma(1)}\bigr)
      \;\wedge\;
      \bigwedge_{j=(n+3)/2}^{n-1} \bigl(-\Vec{r}_{\sigma(j)} + \Vec{r}_{\sigma(1)}\bigr) \\
      \wedge\;
      \bigl(2\Vec{m} - 2\Vec{r}_{\sigma(1)}\bigr)
      \;\Bigr|
    \]
    Multiplying the reflection factors by $-1$, except for the last factor, and removing a factor of $2$ from the last factor, we have this being equal to:
    \[
      2 \left|\;
      \bigwedge_{k=2}^{n-1} \bigl(\Vec{r}_{\sigma(k)} - \Vec{r}_{\sigma(1)}\bigr)
      \;\wedge\;
      \bigl(\Vec{m} - \Vec{r}_{\sigma(1)}\bigr)
      \;\right|
    \]
    We now check that subtracting $\frac1n$ times the first $n-2$ factors from the last factor yields
    \[
      \bigl(\Vec{m} - \Vec{r}_{\sigma(1)}\bigr) \;-\; \frac1n \sum_{k=2}^{n-1} \bigl(\Vec{r}_{\sigma(k)} - \Vec{r}_{\sigma(1)}\bigr)
      \;=\; \frac1n \Vec{r}_{\sigma(n)} \;-\; \frac1n \Vec{r}_{\sigma(1)}.
    \]
    Indeed, writing $\Vec{m} = \frac1n\sum_{k=1}^n \Vec{r}_k$ and tracking the coefficient of each $\Vec{r}_{\sigma(k)}$: the coefficient of $\Vec{r}_{\sigma(1)}$ is $\frac1n - 1 + \frac{n-2}{n} = -\frac1n$, since each of the $n-2$ other factors contributes $-1$ to it; the coefficient of $\Vec{r}_{\sigma(n)}$ is $\frac1n$, since $\Vec{r}_{\sigma(n)}$ appears nowhere among the other factors; and the coefficient of every remaining $\Vec{r}_{\sigma(k)}$, $2 \le k \le n-1$, is $\frac1n - \frac1n = 0$, since each appears with coefficient $+1$ in exactly one of the other factors. We now have the claimed value $\frac1n \Vec{r}_{\sigma(n)} - \frac1n \Vec{r}_{\sigma(1)}$.
    Thus, subtracting $\frac1n$ times the first $n-2$ factors from the last factor shows this determinant is equal to:
    \[
      2 \left|\;
      \bigwedge_{k=2}^{n-1} \bigl(\Vec{r}_{\sigma(k)} - \Vec{r}_{\sigma(1)}\bigr)
      \;\wedge\;
      \Bigl(\frac1n \Vec{r}_{\sigma(n)} - \frac1n \Vec{r}_{\sigma(1)}\Bigr)
      \;\right| \;=\; \frac{2}{n}\bigl(n^{n-2}\bigr) = 2n^{n-3}
    \] by the previous lemma \ref{lem:4.10}.
\end{proof}

With these volume computations, we now have all that we need to prove theorem \ref{thm:4.2}.

\begin{proof}[Proof of Theorem \ref{thm:4.2}]
    By definition, pulling at $r_1$ yields a subdivision of the dihedral permutohedron. From the affine independence criteria, all elements of this subdivision are simplices, so this is a triangulation. For the odd dihedral permutohedron, selecting a facet that doesn't contain $r_1$ or $f_1$ is done exactly by selecting $(n-1)/2$ rotations from the other $n-1$ antipodal pairs. Thus, by lemma \ref{lem:4.11}, we have $\binom{n-1}{(n-1)/2}$ simplices of normalized volume $n^{n-3}$ in the triangulation. For the case of a facet containing $f_1$ but not $r_1$, selecting such a facet first involves choosing one of the other $(n-1)$ antipodal pairs to exclude from the facet. This argument gives $(n-1)\binom{n-1}{(n-1)/2}$ facets that contain $f_1$ or $r_1$. Now, by the antipodal symmetry, half of these facets should contain $f_1$ and not $r_1$, so there are $((n-1)/2)\binom{n-1}{(n-1)/2}$ simplices in the triangulation of normalized volume $2n^{n-3}$ by lemma \ref{lem:4.12}. With this triangulation, we have a normalized volume of 
    \begin{multline}
        n^{n-3}\binom{n-1}{(n-1)/2} + 2n^{n-3}((n-1)/2)\binom{n-1}{(n-1)/2} \\
        = n^{n-3}\binom{n-1}{(n-1)/2} + n^{n-3}(n-1)\binom{n-1}{(n-1)/2} \\
        = n^{n-2}\binom{n-1}{(n-1)/2}
    \end{multline}
    for the odd dihedral permutohedron.
\end{proof}

\begin{remark}
    Certainly, given any vertex $v$ of the dihedral permutohedron there is a matrix \eqref{eq:dihedral_matrices} corresponding to some $w \in D_{2n}$ such that $w(r_1) = v$. Therefore, one could get an equivalent triangulation by pulling at any other vertex of the dihedral permutohedron.
\end{remark}

Let us now perform the pulling triangulation for even $n > 3$. Unlike the odd case, the dihedral permutohedron for even $n$ cannot be triangulated by pulling at a single vertex. We start with the trivial subdivision $\{Q_n\}$. Then, after pulling at $r_1$ \eqref{eq:rotations} this refines to the subdivision $\{\conv(F \cup r_1) \mid r_1 \notin F\}$, where $F$ is a facet of $Q_n$. Let us consider all such facets. The facets not containing $r_1$ are exactly the facets containing $f_1$ \eqref{eq:ref}, as our facets all need to contain exactly one vertex from each antipodal pair. Selecting $f_1$ as one of the reflections in a facet, there are still $n/2$ rotations to be selected from the remaining $n-1$ antipodal pairs to form a facet, yielding $\binom{n-1}{n/2} = \binom{n-1}{n/2 -1} = \binom{n-1}{\lfloor (n-1)/2 \rfloor}$ facets. Thus, our subdivision after pulling at $r_1$ is \begin{equation} \label{eq:pulling}
    \{\conv(F \cup r_1) \mid f_1 \in F\},
\end{equation} where $F$ is a facet of $Q_n$. Since every facet in this set contains $f_1$, refining again by pulling at $f_1$ yields the subdivision \[
\{\conv(F' \cup f_1) \mid f_1 \notin F'\},\] where $F'$ is a facet of some $\conv(F \cup r_1)$ from the previous step. Thus, to understand the pulling subdivision at this point, we need to understand these $F'$ facets. Note that by the affine independence criteria, a facet $F'$ of some $\conv(F \cup r_1)$ such that $f_1 \notin F'$ must be a simplex, as there is no longer the option of forming a facet with $\frac{n}{2}$ reflection vertices. To describe these facets, first note after pulling at $r_1$, every $\conv(F \cup r_1)$ is of the form
\begin{equation}\label{eq:new_facets}
\conv(F \cup r_1) = \conv(r_{\sigma(1)}, r_{\sigma(2)}, \cdots, r_{\sigma(n/2 + 1)}, \, f_{\sigma(n/2 + 2)}, \cdots, f_{\sigma(n)}, f_{\sigma(1)}),
\end{equation}
where $\sigma$ is some permutation such that $\sigma(1) = 1$.

\begin{lemma}
\label{lem:4.13}
    A facet $F'$ of some $\conv(F \cup r_1)$ \eqref{eq:new_facets} from the subdivision $\{\conv(F \cup r_1) \mid r_1 \notin F\}$ in \eqref{eq:pulling} obtained by pulling at $r_1$ such that $F'$ doesn't contain $f_1$ must contain $n/2$ rotations and $n/2 - 1$ reflections.
\end{lemma}
\begin{proof}
     Since we start by excluding $f_1$, a facet excluding $f_1$ must contain $n/2 - 1$ reflection vertices or $n/2 - 2$ reflection vertices. Suppose there is a facet excluding $f_1$ that has $n/2 + 1$ rotations and $n/2 - 2$ reflections. This means there exists a linear inequality $\varphi - c \geq 0$ satisfied by the above polytope such that $\varphi(r_{\sigma(j)}) = c$ for all $j \leq n/2 + 1$ and without loss of generality $\varphi(f_{\sigma(j)}) = c$ for all $n/2 + 1 < j < n$. Furthermore, these equalities cannot hold for the other $2$ reflections. With this in mind, \eqref{eq:ref} gives $\varphi(f_{\sigma(j)}) = \varphi(2/n(r_1 + \cdots + r_{n}) - r_{\sigma(j)}) = c$ for all $n/2 + 1 < j < n$. This means $\varphi$ is equal to some constant $d$ for these other rotation vertices. Computing
    \[
    \varphi\left(\frac{2}{n}(r_1 + \cdots + r_{\sigma(n)})\right) = \frac{2}{n}\left(\left(\frac{n}{2}+1\right)c + \left(\frac{n}{2}-2\right)d + \varphi(r_{\sigma(n)})\right).
    \]
    Since $\varphi(2/n(r_1 + \cdots + r_{\sigma(n)})) - d = c$ if and only if $\varphi(2/n(r_1 + \cdots + r_{\sigma(n)})) = c + d$, we have that\[
    2/n(c - 2d + \varphi(r_{\sigma(n)})) = 0 \quad \text{if and only if} \quad \varphi(r_{\sigma(n)}) = 2d - c.
    \] Since $f_{\sigma(n)} = 2/n(r_1 + \cdots + r_{\sigma(n)}) - r_{\sigma(n)}$ by \eqref{eq:ref}, we have that $\varphi(f_{\sigma(n)}) = (c+d) - (2d-c) = 2c - d$, and so $\varphi(f_{\sigma(n)}) - c = c - d$. This expression is positive, since the linear inequality is satisfied by all vertices of the polytope. However, since $f_{\sigma(1)} = f_1$ is a vertex of the polytope but is not in the facet, we also have $\varphi(f_{\sigma(1)}) - c = d - c > 0$, which contradicts $c-d>0$. Therefore, a facet must contain $n/2$ rotations and $n/2 - 1$ reflections. \end{proof}

\begin{lemma}
\label{lem:4.14}
    If a facet of some $\conv(F \cup r_1)$ \eqref{eq:new_facets} from the pulling subdivision $\{\conv(F \cup r_1) \mid r_1 \notin F\}$ in \eqref{eq:pulling} doesn't contain $f_1$, then it contains $r_1$.
\end{lemma}

\begin{proof}
    With the vertices of $\conv(F \cup r_1)$ in \eqref{eq:new_facets}, suppose there is a facet that doesn't contain $f_1$ and doesn't contain $r_1$. The previous lemma \ref{lem:4.13} shows this facet must contain $n/2$ rotations and $n/2 - 1$ reflections. Thus, the only possibility for the vertices of such a facet is to choose all the other vertices besides $f_1$ and $r_1$. Then, there is a linear inequality $\varphi - c \geq 0$ that holds for the $n+1$ vertices above and achieves equality for all of the above vertices except for $r_1$ and $f_1$. Thus, for all $j > n/2 + 1$, $\varphi(r_{\sigma(j)}) = d$, some other constant. With this in mind, we have\[
    \varphi\bigl((n+1)\mathbf{1}\bigr) = c + d = \frac{2}{n}\left(\frac{n}{2}c + \left(\frac{n}{2}-1\right)d + \varphi(r_{\sigma(1)})\right).
    \] This equation is equivalent to $\varphi(r_{\sigma(1)}) = d$, but by \eqref{eq:ref} this means $\varphi(f_{\sigma(1)}) = c + d - d = c$, so $f_{\sigma(1)} = f_1$ should be a vertex of the facet, a contradiction.\end{proof}

\begin{lemma}
\label{lem:4.15}
    Any set of $n-1$ vertices taken from some $\conv(F \cup r_1)$ \eqref{eq:new_facets} from the subdivision $\{\conv(F \cup r_1) \mid r_1 \notin F\}$ in \eqref{eq:pulling} that excludes $f_1$ and excludes a rotation vertex besides $r_1$ forms a facet.
\end{lemma}

\begin{proof}
    Without loss of generality, choose the following vertices for the potential facet: \[
    r_{\sigma(1)}, \cdots, r_{\sigma(n/2)}, f_{\sigma(n/2 + 2)}, \cdots, f_{\sigma(n)}.
    \] Consider the linear functional $\varphi_\sigma$ (\ref{def:linear_functional}). We have that $\varphi_\sigma((n+1)\mathbf{1}) = 2/n(n/2) = 1$, so \eqref{eq:ref} gives $\varphi_\sigma(f_{\sigma(j)}) = 0$ for all $j \geq n/2 + 2$. Indeed, this linear functional vanishes on the desired $n-1$ vertices. Additionally, by \eqref{eq:ref}, $\varphi_\sigma(f_{\sigma(1)}) = 1 = \varphi_\sigma(r_{\sigma(n/2 + 1)})$, so the linear inequality $\varphi_\sigma(v) \geq 0 $ vanishes only on the desired vertices and is a valid inequality of $\conv(F \cup r_1)$.\end{proof}

The above $3$ lemmas fully characterize the facets $F'$ of the $\conv(F \cup r_1)$ \eqref{eq:new_facets} that do not contain $f_1$. For each $\conv(F \cup r_1)$ in the previous subdivision, all the facets excluding $f_1$ come from choosing a rotation vertex of $\conv(F \cup r_1)$ besides $r_1$ to exclude. Thus, each $\conv(F \cup r_1)$ has $n/2$ facets not containing $f_1$, as there are $n/2$ rotation vertices in $F$. Since there were $\binom{n-1}{\lfloor (n-1)/2 \rfloor}$ elements in the previous subdivision, we have $\frac{n}{2}\binom{n-1}{{\lfloor (n-1)/2 \rfloor}}$ elements in the subdivision after pulling at $f_1$.

\begin{lemma}
\label{lem:4.16}
    The subdivision after pulling at $r_1$ and $f_1$ is now a triangulation with each simplex having normalized volume $2n^{n-3}$.
\end{lemma}

\begin{proof}
    From the previous lemmas, we have that each $\conv(F' \cup f_1)$ has vertices \[
    r_{\sigma(1)}, \cdots, r_{\sigma(n/2)}, f_{\sigma(n/2 + 2)}, \cdots, f_{\sigma(n)}, f_{\sigma(1)}, \] where $\sigma(1) = 1$. This is a set with $n$ vertices and $1$ antipodal pair, which by the criteria for affine independence, forms a simplex. Hence, we now have a triangulation after pulling at $r_1, f_1$. Lastly, applying the same method in the proof of lemma \ref{lem:4.12} to these simplices shows that these simplices all have normalized volume $2n^{n-3}$.
\end{proof}

\begin{proof}[Proof of Theorem \ref{thm:4.3}]
    By Lemma \ref{lem:4.16}, the subdivision obtained by pulling at $r_1$ and then $f_1$ is a triangulation in which every simplex has normalized volume $2n^{n-3}$. By lemmas \ref{lem:4.13}--\ref{lem:4.15}, each of the $\binom{n-1}{\lfloor (n-1)/2 \rfloor}$ facets $\conv(F \cup r_1)$ contributes exactly $n/2$ such simplices, so the triangulation has $(n/2)\binom{n-1}{\lfloor(n-1)/2 \rfloor}$ simplices in total. Hence the normalized volume of the even dihedral permutohedron is:
    \[
    2n^{n-3}\left(\frac{n}{2}\right)\binom{n-1}{\lfloor \frac{n-1}{2} \rfloor} = n^{n-2}\binom{n-1}{\lfloor \frac{n-1}{2}\rfloor}.
    \]
\end{proof}

\begin{remark}
    Certainly, one could map the antipodal pair $(r_1, f_1)$ to any other antipodal pair $(r_i, f_i)$ by applying the matrix corresponding to some rotation $w \in D_{2n}$.
    Thus, one could produce an equivalent triangulation by pulling at any other two vertices in an antipodal pair. 
\end{remark}

\section{A Nontrivial Affine Symmetry}

For any orbit polytope corresponding to some group of permutation matrices $G$, we have an affine symmetry through applying any matrix $g \in G$ to the orbit polytope. In \cite{MR3436389}, the main question that is considered is when an orbit polytope has more affine symmetries than the group $G$. In this section, I construct a unimodular affine symmetry of the odd dihedral permutohedron that isn't one of the matrices $w \in D_{2n}$, when considering $D_{2n}$ as a group of permutation matrices. Before giving the theorem, consider the following definition.

\begin{definition}
    Two lattice polytopes $P,Q$ that live in $\mathbb{R}^n$ are \emph{unimodularly equivalent} if there exists an affine transformation $Ax + b$ that maps $P$ to $Q$ with $A \in \text{GL}_n(\mathbb{Z}), b \in \mathbb{Z}^n$.
\end{definition}

\begin{theorem}
\label{thm:5.3}
    There is a unimodular affine symmetry of the odd dihedral permutohedron that permutes the rotation vertices \eqref{eq:rotations} and the reflection vertices \eqref{eq:ref} by
    \[
        r_j \mapsto r_{2j-1 \,(\text{mod } n)}, \qquad f_j \mapsto f_{2j-1 \,(\text{mod } n)}, \qquad j \in [n].
    \]
    Here, the representative of the identity element in $\mathbb{Z}/n\mathbb{Z}$ is $n$.
\end{theorem}

\begin{example}
    Let $n = 5$. The map $j \mapsto 2j - 1 \pmod 5$ sends $(1,2,3,4,5)$ to $(1,3,5,2,4)$. Hence the proposed affine symmetry induces the bijection
    \[
        r_1 \mapsto r_1,\ r_2 \mapsto r_3,\ r_3 \mapsto r_5,\ r_4 \mapsto r_2,\ r_5 \mapsto r_4,
    \]
    and the same permutation of indices holds for $f_1, \dots, f_5$.
\end{example}

Since the $n$ rotation vertices are linearly independent, there is a unique invertible $n \times n$ linear map $L_n$ with $L_n(r_j) = r_{2j-1}$ for all $j \in [n]$.

\begin{lemma}
\label{lem:5.4}
    The linear map $L_n$ with respect to the standard basis is
    \[L_n \;=\;
\begin{pmatrix}
1      &        &        &        & 1      &        &        &        \\
1      &        &        &        &        & 1      &        &        \\
       & 1      &        &        &        & 1      &        &        \\
       & 1      &        &        &        &        & 1      &        \\
       &        & \ddots &        &        &        & 1       & \ddots \\
       &        &        & 1      &        &        &        &      \\
       &        &        & 1      & \ddots &        &        & 1        \\
       &        &        &        & 1      &        &         & 1
\end{pmatrix} - \tfrac{1}{n}J_n, \qquad J_n = n \times n \ \text{all-ones matrix}.
\]
\end{lemma}

In the matrix above, all entries besides the $1$'s are $0$'s and given column $i$, column $i+1$ is the result of shifting the position of entries of column $i$ by the map $x \mapsto x + 2 \pmod{n}$. Let

\begin{equation} \label{eq:pn}
    P_n = \begin{pmatrix}
1      &        &        &        & 1      &        &        &        \\
1      &        &        &        &        & 1      &        &        \\
       & 1      &        &        &        & 1      &        &        \\
       & 1      &        &        &        &        & 1      &        \\
       &        & \ddots &        &        &        & 1       & \ddots \\
       &        &        & 1      &        &        &        &      \\
       &        &        & 1      & \ddots &        &        & 1        \\
       &        &        &        & 1      &        &         & 1
\end{pmatrix}.
\end{equation}
To prove Lemma~\ref{lem:5.4}, we can first try to understand $P_n$.

\begin{lemma}
\label{lem:5.5}
    Each row of $P_n$ \eqref{eq:pn} has exactly $2$ nonzero entries, both equal to $1$. Using $1$-indexing, row $i = 2m$ has its $1$'s in columns $m$ and $\tfrac{n+1}{2} + m$.
    Row $i = 2m+1$ has its $1$'s in columns $m+1$ and $\tfrac{n+1}{2} + m$.
\end{lemma}
\begin{proof}
    We treat the two parities of $i$ separately.
    \begin{itemize}
        \item If $i = 2m$, then row $i$ contains a $1$ coming from the lower entry of a pair in the first half of the columns of $P_n$, which lies in column $m$, and a $1$ coming from the upper entry of a pair in the second half of the columns, which lies in column $\tfrac{n+1}{2} + m$.
        \item If $i = 2m+1$, then row $i$ contains a $1$ coming from the upper entry of a pair in the first half of the columns, which lies in column $m+1$, and a $1$ coming from the lower entry of a pair in the second half of the columns, which lies in column $\tfrac{n+1}{2} + m$.
    \end{itemize}
\end{proof}

\begin{proof}[Proof of Lemma~\ref{lem:5.4}]
    Fix $j \in [n]$ and put $k = j - 1 \in [0, n-1]$, so that $(r_j)_c \equiv c + k \pmod n$. It suffices to show $[L_n r_j]_i \equiv i + 2k \pmod n$, since $(r_{2j-1})_i = (r_{2k+1})_i \equiv i + 2k$. Each coordinate of $r_j$ lies in $[n]$ and $\tfrac1n J_n r_j = \tfrac{n+1}{2}\mathbf{1}$, so
    \[
        [L_n r_j]_i = [P_n r_j]_i - \tfrac{n+1}{2}.
    \]
    By Lemma~\ref{lem:5.5}, $[P_n r_j]_i$ is the sum of two coordinates of $r_j$. We reduce each into $[n]$, add, and subtract $\tfrac{n+1}{2}$. We treat $i$ even and $i$ odd separately.

    \medskip
    \noindent \textbf{Case $i = 2m$.} Here $m \in \left[1, \tfrac{n-1}{2}\right]$, and the two contributions are $(m + k) \pmod n$ and $\left(m + \tfrac{n+1}{2} + k\right) \pmod n$. There are three cases:
    \begin{itemize}
        \item If $m + \tfrac{n+1}{2} + k \leq n$, meaning that $i + 2k \leq n - 1$, then the $i$th coordinate is $2m + 2k = i + 2k$.
        \item If $m + k \leq n < m + \tfrac{n+1}{2} + k$, i.e., $n < i + 2k \leq 2n$, then the $i$th coordinate is $2m + 2k - n = i + 2k - n$.
        \item If $n < m + k$, i.e., $2n < i + 2k \leq 3n$, then the $i$th coordinate is $2m + 2k - 2n = i + 2k - 2n$.
    \end{itemize}

    \medskip
    \noindent \textbf{Case $i = 2m+1$.} Here $m \in \left[0, \tfrac{n-1}{2}\right]$, and the two contributions are $(m + 1 + k) \pmod n$ and $\left(m + \tfrac{n+1}{2} + k\right) \pmod n$. There are three cases:
    \begin{itemize}
        \item If $m + \tfrac{n+1}{2} + k \leq n$, i.e., $i + 2k \leq n$, then the $i$th coordinate is $2m + 1 + 2k = i + 2k$.
        \item If $m + 1 + k \leq n < m + \tfrac{n+1}{2} + k$, i.e., $n < i + 2k \leq 2n$, then the $i$th coordinate is $i + 2k - n$.
        \item If $n < m + 1 + k$, i.e., $2n < i + 2k \leq 3n$, then the $i$th coordinate is $i + 2k - 2n$.
    \end{itemize}

    In every case, the $i$th coordinate is congruent to $i + 2k \pmod n$ and lies in $[n]$. Hence $L_n r_j = r_{2j-1}$.
\end{proof}

While $L_n$ induces the desired bijection of the $n$ rotation vertices, it remains to show that $L_n$ induces the desired bijection of the $n$ reflection vertices.

\begin{lemma}
\label{lem:reflections}
    The linear map $L_n$ satisfies $L_n(f_j) = f_{2j-1 \,(\text{mod } n)}$ for all $j \in [n]$.
\end{lemma}

\begin{proof}
    Each row of $P_n$ has exactly two $1$'s by Lemma~\ref{lem:5.5}, so $P_n(\mathbf{1}) = 2 \cdot \mathbf{1}$. Also $\tfrac1n J_n \mathbf{1} = \mathbf{1}$. Hence $L_n \mathbf{1} = \mathbf{1}$. By \eqref{eq:ref} and Lemma~\ref{lem:5.4},
    \[
        L_n(f_j) = L_n\left((n+1)\mathbf{1} - r_j\right) = (n+1)\mathbf{1} - r_{2j-1} = f_{2j-1}. \qedhere
    \]
\end{proof}

The last thing we need to show to prove theorem \ref{thm:5.3} is that this linear map is unimodular when considered as a map in the natural $(n-1)$-dimensional space considered earlier in this paper.

\begin{proof}[Proof of Theorem \ref{thm:5.3}]
To project this map $L_n$ down to the $(n-1)$-dimensional space and show that it is unimodular, we check what it does to the $n-1$ simple-root basis vectors $e_1-e_2, e_2-e_3, \ldots, e_{n-1}-e_n$. Since $L_n = P_n - \frac{1}{n}J_n$ and $J_n$ maps each of these basis vectors to $0$, we have $L_n(e_{i-1}-e_i) = P_n(e_{i-1}-e_i)$. In particular, $L_n(e_{i-1}-e_i)$ has integer coordinates and coordinate sum $0$, so it has integer coordinates when expressed in the simple-root basis $\{e_{j-1}-e_j\}$. Hence, $L_n$ is an integer matrix after projection. Since the $n$ rotation vertices are linearly independent in $\mathbb{R}^n$, after translating by $-r_1$ and embedding into $(n-1)$-dimensional space, the other $n-1$ rotation vertices form an $\mathbb{R}$-basis of the $(n-1)$-dimensional space. Since $L_n(r_1) = r_1$, the projected map still has the same action on the projected vertices as it does on the original vertices. Since $r_2-r_1, r_3-r_1, \ldots, r_n-r_1$ form an $\mathbb{R}$-basis of the $(n-1)$-dimensional space and our map $L_n$ permutes these basis elements, our projected matrix is invertible and its inverse is some power of itself. Since any power of an integer matrix is an integer matrix, $L_n$ is an invertible integral matrix with an integer matrix inverse, so $L_n \in GL_{n-1}(\mathbb{Z})$ and is hence a unimodular map.\end{proof}

Below is an application of this interesting affine symmetry of the dihedral permutohedron.

\begin{cor}
    With $s_1, s_2$ defined in \eqref{eq:odd_perm}, consider the subpolytopes $P_n, R_n \subset Q_n$, where $P_n = \conv(D_{2n}\rho - \{r_1, s_2(\rho)\})$ and $R_n = \conv(D_{2n}\rho - \{r_1, s_1(\rho)\})$. $P_n$ and $R_n$ are unimodularly equivalent for odd $n \geq 3$.
\end{cor}

\begin{proof}
By theorem \ref{thm:5.3}, $L_n$ maps $r_1 \mapsto r_1$ and maps $s_2(\rho) = f_n \mapsto f_{n-1} = s_1(\rho)$ and so $P_n$ maps to $R_n$ under this unimodular equivalence. Note that unimodular equivalence is an even stronger condition than Ehrhart equivalence.
\end{proof}

Lastly, here are some of the interesting questions that this paper did not address.

\section{Future Questions}

\begin{question}
    Can we describe the number of lattice points or the Ehrhart polynomial of the dihedral permutohedron?
\end{question}

\begin{question}
    What is the full group of affine symmetries for the dihedral permutohedron?
\end{question}

\begin{question}
    What happens if we define the dihedral permutohedron using an arbitrary basepoint or ``generic'' basepoint?
\end{question}

\begin{question}
    Can we use the arguments of this paper to calculate the normalized volume of other centrally symmetric polytopes?
\end{question}

\section{Acknowledgements}

I am very thankful to Yuhan Jiang for giving me the original problem that inspired this paper, advising me throughout my research this summer, and for helping me make many revisions to this paper! Also many thanks to Aylin Ambriz for advice on deciding the title and creating the dihedral permutohedron figures for $n=3,4$.

\bibliographystyle{alpha}
\bibliography{main}

@book{schoute,
author = {Schoute, P. H. and Cardinaal, Jacob},
address = {Amsterdam},
booktitle = {Analytical treatment of the polytopes regularly derived from the regular polytopes.},
language = {eng},
lccn = {a  29000397},
publisher = {J. M\"uller},
series = {Verhandelingen der Koninklijke akademie van wetenschappen te Amsterdam. 1. sectie, deel XI, no. 3, 5; deel XII, no. 2},
title = {Analytical treatment of the polytopes regularly derived from the regular polytopes. By P.H. Schoute.},
year = {1911 - 1916},
}

@article {MR45168,
    AUTHOR = {Rado, R.},
     TITLE = {An inequality},
   JOURNAL = {J. London Math. Soc.},
  FJOURNAL = {The Journal of the London Mathematical Society},
    VOLUME = {27},
      YEAR = {1952},
     PAGES = {1--6},
      ISSN = {0024-6107,1469-7750},
   MRCLASS = {27.0X},
  MRNUMBER = {45168},
MRREVIEWER = {R.\ C.\ Lyndon},
       DOI = {10.1112/jlms/s1-27.1.1},
       URL = {https://doi.org/10.1112/jlms/s1-27.1.1},
}

@article {MR1239510,
    AUTHOR = {Onn, Shmuel},
     TITLE = {Geometry, complexity, and combinatorics of permutation
              polytopes},
   JOURNAL = {J. Combin. Theory Ser. A},
  FJOURNAL = {Journal of Combinatorial Theory. Series A},
    VOLUME = {64},
      YEAR = {1993},
    NUMBER = {1},
     PAGES = {31--49},
      ISSN = {0097-3165,1096-0899},
   MRCLASS = {52B12 (05E10)},
  MRNUMBER = {1239510},
MRREVIEWER = {Egon\ Schulte},
       DOI = {10.1016/0097-3165(93)90086-N},
       URL = {https://doi.org/10.1016/0097-3165(93)90086-N},
}

@article {MR3029450,
    AUTHOR = {Baumeister, Barbara and Haase, Christian and Nill, Benjamin
              and Paffenholz, Andreas},
     TITLE = {Polytopes associated to dihedral groups},
   JOURNAL = {Ars Math. Contemp.},
  FJOURNAL = {Ars Mathematica Contemporanea},
    VOLUME = {7},
      YEAR = {2014},
    NUMBER = {1},
     PAGES = {30--38},
      ISSN = {1855-3966,1855-3974},
   MRCLASS = {52B12 (05E10 20B35 52B05 52B20)},
  MRNUMBER = {3029450},
MRREVIEWER = {Alexander\ Kolpakov},
       DOI = {10.26493/1855-3974.289.91d},
       URL = {https://doi.org/10.26493/1855-3974.289.91d},
}

@article {MR2538016,
    AUTHOR = {Baumeister, Barbara and Haase, Christian and Nill, Benjamin
              and Paffenholz, Andreas},
     TITLE = {On permutation polytopes},
   JOURNAL = {Adv. Math.},
  FJOURNAL = {Advances in Mathematics},
    VOLUME = {222},
      YEAR = {2009},
    NUMBER = {2},
     PAGES = {431--452},
      ISSN = {0001-8708,1090-2082},
   MRCLASS = {52B20 (05E15 52B05)},
  MRNUMBER = {2538016},
MRREVIEWER = {Benjamin\ Braun},
       DOI = {10.1016/j.aim.2009.05.003},
       URL = {https://doi.org/10.1016/j.aim.2009.05.003},
}

@article {MR3436389,
    AUTHOR = {Friese, Erik and Ladisch, Frieder},
     TITLE = {Affine symmetries of orbit polytopes},
   JOURNAL = {Adv. Math.},
  FJOURNAL = {Advances in Mathematics},
    VOLUME = {288},
      YEAR = {2016},
     PAGES = {386--425},
      ISSN = {0001-8708,1090-2082},
   MRCLASS = {52B12 (05E15 20B25 20C15 52B15)},
  MRNUMBER = {3436389},
MRREVIEWER = {P.\ McMullen},
       DOI = {10.1016/j.aim.2015.10.021},
       URL = {https://doi.org/10.1016/j.aim.2015.10.021},
}

@article {stanley80,
    AUTHOR = {Stanley, Richard P.},
     TITLE = {Decompositions of rational convex polytopes},
   JOURNAL = {Ann. Discrete Math.},
  FJOURNAL = {Annals of Discrete Mathematics},
    VOLUME = {6},
      YEAR = {1980},
     PAGES = {333--342},
   MRCLASS = {52A43},
  MRNUMBER = {593545},
MRREVIEWER = {P.\ McMullen},
}

@book{circ,
    author = {Davis, Philip J.},
    address = {New York},
    title = {Circulant matrices},
    isbn = {0471057711},
    publisher = {Wiley},
    series = {Pure and applied mathematics},
    year = {1979},
}
\end{document}